\documentclass[12pt]{amsart}

\newcommand{\GAP}{{\sf GAP}}
\newcommand{\KBMAG}{{\sf KBMAG}}
\newcommand{\N}{{\mathbb N}}
\newcommand{\Z}{{\mathbb Z}}
\newcommand{\Geo}{\mathrm{Geo}}
\newcommand{\pre}{\mathrm{pre}}
\newcommand{\suf}{\mathrm{suf}}
\newcommand{\sub}{\mathrm{sub}}

\newcommand{\supp}{\mathrm{supp}}

\newcommand{\GW}{\mathrm{GW}}
\newcommand{\geo}{\Geo(G,X)}
\newcommand{\LT}{\mathrm{LT}}

\newcommand{\SSgp}{\mathrm{SS}}

\newcommand{\ifof}{if and only if\ }
\newcommand{\te}{\!=\!}

\def\klt#1{#1-locally testable}
\def\klty#1{#1-local testability}

\usepackage{amsmath,amssymb}
\newtheorem{theorem}{Theorem}[section]

\newtheorem{proposition}[theorem]{Proposition}
\newtheorem{lemma}[theorem]{Lemma}

\newtheorem{corollary}[theorem]{Corollary}

\newenvironment{mylist}{\begin{list}{}{
\setlength{\parskip}{0mm}
\setlength{\topsep}{2mm}
\setlength{\parsep}{0mm}
\setlength{\itemsep}{0mm}
\setlength{\labelwidth}{7mm}
\setlength{\labelsep}{3mm}
\setlength{\itemindent}{0mm}
\setlength{\leftmargin}{12mm}
\setlength{\listparindent}{6mm}
}}{\end{list}}

\begin{document}
\title{Groups whose geodesics are locally testable}

\author[S.~Hermiller]{Susan Hermiller}
\address{Department of Mathematics\\
        University of Nebraska\\
         Lincoln NE 68588-0130, USA}
\email{smh@math.unl.edu}

\author[Derek F.~Holt]{Derek F.~Holt}
\address{Mathematics Institute\\
      University of Warwick  \\
       Coventry CV4 7AL, UK  }
\email{dfh@maths.warwick.ac.uk}

\author[Sarah Rees]{Sarah Rees}
\address{School of Mathematics and Statistics\\
       University of Newcastle \\
       Newcastle NE1 7RU, UK  }
\email{Sarah.Rees@ncl.ac.uk}


\begin{abstract}
A regular set of words is ($k$-)locally testable if membership of a word
in the set is determined by the nature of its subwords of some bounded
length $k$. In this article we study groups for which the set of all geodesic
words with respect to some generating set is ($k$-)locally testable,
and we call such groups ($k$-)locally testable. 
We show that a group is \klt{1} if and only if it is free abelian.
We show that the class of ($k$-)locally testable groups
is closed under taking finite direct products.
We show also that a locally testable group
has finitely many conjugacy classes of torsion elements.

Our work involved computer investigations of specific groups, for which
purpose we implemented an algorithm in \GAP\ to compute a finite state automaton
with language equal to the set of all geodesics of a group (assuming that this
language is regular), starting from a shortlex automatic structure.
We provide a brief description of that algorithm.
\end{abstract}

\maketitle


\section{Introduction}\label{sec:intro}

Let $G$ be a finitely generated group and let $X$ be a finite symmetric
(that is, inverse-closed) generating set of $G$.
Let $\Geo(G,X)$ be the set of all
words over $X$ that label geodesic paths in the corresponding
Cayley graph for $G$.  For many classes of groups, including
word hyperbolic groups~\cite[Theorem 3.4.5]{ECHLPT},
abelian groups, geometrically finite hyperbolic groups
~\cite{NeumannShapiro}, 
Coxeter groups~\cite{How93},
and Garside groups~\cite{CharneyMeier}, there are generating
sets $X$ for which the language $\Geo(G,X)$ is regular.

Starting only from the assumption that $\Geo(G,X)$ is regular,
we know very little about $G$. We do however have the following result,
which applies in the more general situation when $\Geo(G,X)$ is recursive.

\begin{proposition} With the notation above, if $\Geo(G,X)$ is recursive and
$G$ has a recursively enumerable presentation, then $G$ has solvable
word problem.
\end{proposition}
\begin{proof}
Let $w$ be a word over $X$. Since $\Geo(G,X)$ is recursive, we can decide
whether or not $w \in \Geo(G,X)$. If so, then $w =_G 1$ \ifof $w$ is the
empty word. If not, then there exists a word $v$ with $v =_G w$ and
$l(v) < l(w)$, and so $v^{-1}w$ is a relator of $G$. But since
$G$ has a recursively enumerable presentation, we can enumerate the relators
of $G$ and hence find such a $v$. We now repeat the process with $v$ in place 
of $w$.
\end{proof}

We are motivated by the question of whether there might be interesting
subclasses of the class of regular languages for which we could
say more about groups with $\Geo(G,X)$ in that subclass.
In~\cite{HermillerHoltRees}, we considered groups for which $\Geo(G,X)$ is
a {\em star-free} regular language
(see also ~\cite[Chapter 4, Definition 2.1]{Pin} for the definition),
and although we obtained some conditions under
which groups have this property, we were not able to obtain any
classifications. In~\cite{GHHR}, we considered groups satisfying the much
more restrictive hypothesis that $\Geo(G,X)$ is {\em locally excluding} for
some symmetric generating set $X$ of $G$, which means that there exists
a finite set $\mathcal{W}$ of words over $X$ with the property that
a word $w \in \Geo(G,X)$ \ifof $w$ does not have a subword equal to
a word in  $\mathcal{W}$. We proved that this is the case
\ifof $G$ is virtually free.

In this paper, we consider groups $G$ for which $\Geo(G,X)$ is a
{\em locally testable} language for some $X$. The locally testable
languages form a class 
lying between locally excluding and star-free regular languages.

Informally, where $k$ is a positive integer, a set of words is
\klt{$k$} if membership of a word in the set depends
on the nature of its subwords of length $k$.
By a {\em subword} of a word $a_1a_2\cdots a_n$, we mean either the
empty word or a contiguous substring  $a_i a_{i+1}\cdots a_j$ for some
$1 \le i \le j \le n$. 

Examples of locally testable sets are given by the sets of all geodesics
in both the free group on 2 generators and the free abelian 
group on 2 generators, using standard generating sets. 
The set of geodesics in the free group on generators $a,b$ is
\klt{2}, since a word in those generators is
geodesic if and only if it contains no subword equal to
any of $aa^{-1}$, $a^{-1}a$, $bb^{-1}$ or $b^{-1}b$.
The set of geodesics in the free abelian group on generators
$a,b$ is \klt{1} since a word in those
generators is geodesic if and only if it 
contains neither both $a$ and $a^{-1}$ nor both $b$ and $b^{-1}$.
Note that the first example is locally excluding but the
second is not. We can show similarly that the geodesics of
free groups and free abelian groups
on any number of generators are respectively 2- and \klt{1}.

A formal definition of a locally testable language is given in Section
\ref{sec:char}. 
We shall see that
membership of a word $w$ in such a language may
also depend on the prefix and suffix of the word of length $k-1$
as well as its subwords of length $k$. In other words, it depends on the
subwords of length $k$ in the word $\diamond w \diamond$, where $\diamond$ is
a symbol not lying in $X$.

From a language theoretic point of view, the locally testable
languages have been well-studied. Various characterizations,
provided by Brzozowksi and Simon, and by McNaughton,
are described in Section \ref{sec:char}.

The local testability of $\Geo(G,X)$ will certainly depend
on $X$ in general. For example, the free group on four generators has
$\Geo(G,X)$ locally excluding and hence locally testable when $X$ is a
set of free generators for $G$, but in~\cite{HermillerHoltRees} we
exhibited a generating set of size 6 for which $\Geo(G,X)$ is not star-free
and hence not locally testable.

We shall say that the group $G$ is \klt{$k$} if $\Geo(G,X)$ is \klt{$k$} for some
generating set $X$ of $G$, and $G$ is locally testable if it is \klt{$k$} for
some $k$. 

The 2-generator Artin groups $A_k$ are defined by the presentations
$$\langle\, a,b \mid (ab)^{k/2} = (ba)^{k/2}\, \rangle\ (k\ \mathrm{even}),$$
$$\langle\, a,b \mid (ab)^{(k-1)/2}a = (ba)^{(k-1)/2}b \, \rangle\ 
(k\ \mathrm{odd}).$$
It is proved in~\cite{MM06} that $\Geo(G,X)$ is regular with
$X = \{a,b,a^{-1},b^{-1}\}$. In fact it follows
directly from Proposition 4.3 of that paper that $\Geo(G,X)$ is \klt{$k$},
so this provides a further series of examples of
\klt{$k$} groups.

It can be shown that the kernel of the natural homomorphism of $A_n$ onto
the dihedral group of order $2n$ in which the images of $a$ and $b$ have
order 2 is a direct product of an infinite cyclic group
and a free group of rank $k-1$.  We conjecture that any locally testable
group is virtually a direct product of free groups.

We are able to prove a complete characterization of \klt{1} groups.
Specifically, in Section \ref {sec:1test}, we show the following.

\medskip

\noindent{\bf Theorem \ref{thm:1testfreeab}.} {\it
Let $G$ be a finitely generated group.  Then the following are equivalent.
\begin{enumerate}
\item $G$ is free abelian.
\item $G$ is \klt{1}.
\item There is a finite symmetric generating set $X$ for $G$
such that the syntactic semigroup associated to $\Geo(G,X)$ is idempotent.
\end{enumerate}
}

\medskip

In Section~\ref{sec:dirprod},
we prove the following.

\medskip

\noindent {\bf Theorem \ref{thm:ltprod}.} {\it
For any $k > 0$ the class of \klt{$k$} groups
is closed under taking finite direct products.
}

\medskip

In fact, this is deduced as a corollary of the more general result:

\medskip

\noindent {\bf Theorem \ref{thm:genprod}.} {\it
Let $\mathcal{L}$ be a class of languages that is closed under the
operations of union and taking inverse images under length preserving morphisms.
Then the class of $\mathcal{L}$-groups is closed under taking finite direct
products.
}

\medskip

In Section~\ref{sec:lttorconj}, we show that the local testability
restriction on the geodesics results in
an algebraic restriction on the group.

\medskip

\noindent{\bf Theorem \ref{thm:lttorconj}.}  {\it
A locally testable group
has finitely many conjugacy classes of torsion elements.
}

\medskip

Finally, in Section~\ref{sec:geowa}, we
describe a method which, upon input of a finite presentation
of a shortlex automatic group $G$ with symmetric generating 
set $X$, will attempt to construct a finite
state automaton with language equal to $\Geo(G,X)$.  This process is
guaranteed to succeed eventually if $\Geo(G,X)$ is a regular language. 
This has been implemented using the \GAP~\cite{GAP} interface 
to \KBMAG~\cite{KBMAG}.  

Using our program to construct a minimal automaton
for $\Geo(G,X)$ when this language is regular, we
can then use conditions equivalent to those
described in Section~\ref{sec:char} to
decide whether $\Geo(G,X)$ is also ($k$-)locally testable.
These conditions, as well as our results, are stated
throughout this paper in terms of the syntactic semigroup
of the regular language, but for computational purposes 
it is generally more efficient to work with the 
the transition semigroup of the minimal automaton $M$;
these two semigroups are isomorphic (see for example
\cite[p.~18]{Pin}).  It is shown in \cite{Trahtman}
that local testability can be tested in polynomial time from the input $M$.
We conclude the paper by listing some examples of the use of our program 
to test $\Geo(G,X)$ first for being regular and then for being
($k$-)locally testable.



\section{Definition and characterizations of local testability} \label{sec:char}


Let $X$ be any finite set.
Let $X^*$ be the free monoid over $X$, that is the set of
all strings over $X$, and let $X^+$ be the free semigroup,
the set of all non-empty strings.

Let $k > 0$ be a natural number.
For $u \in X^*$ of length at least $k$, let $\pre_k(u) $ be the prefix
of $u$ of length $k$, let $\suf_k(u) $ be the suffix
of $u$ of length $k$, and let $\sub_k(u) $ be the set
of all subwords of $u$ of length $k$.  If $l(u)<k$,
then we define $\pre_k(u)=u$, $\suf_k(u)=u$, and 
$\sub_k(u)=\emptyset$.
Two words $u,v \in X^*$ satisfy $u \sim_k v$ \ifof
$\pre_{k-1}(u)=\pre_{k-1}(v)$, $\suf_{k-1}(u)=\suf_{k-1}(v)$, 
and $\sub_{k}(u)=\sub_{k}(v)$.
A subset $L \subseteq X^*$
is defined to be {\it \klt{$k$}} \cite[p.~247]{BrzozowskiSimon}
if $L$ is a union of 
equivalence classes of $\sim_k$.  The set $L$ is called
{\it locally testable} if $L$ is \klt{$k$} for some $k$.

Let $\LT$ denote the class of locally testable languages and $\LT(k)$ the class
of all \klt{$k$} languages.
From the definition as a union of equivalence classes, it follows immediately
that $\LT(k)$ is closed under the Boolean operations of
union, intersection and complementation.  It also follows from
the definition that any \klt{$k$} language
is also \klt{$m$} for all $m \ge k$, and hence $\LT$ is also
closed under the three Boolean operations.

In Proposition \ref{prop:ltbool} we list two characterizations of
local testability due to Brzozowski and Simon, and McNaughton.

From the first,
it follows immediately that $\LT$ is strictly contained in
the class of regular languages over $X$.
(In fact it shows that locally testable languages are star-free.)

The second characterization involves the
syntactic semigroup of $L$, $\SSgp(L)$, which is defined to be
the quotient  \mbox{$X^+/\!\sim_L$}, where 
$\sim_L$ is the congruence defined as follows.
We have $u \sim_L v$ if and only if, for all words $s$ and $t \in X^*$,
$sut$ and $svt$ are either both in $L$ or both not in $L$.
Note that the syntactic semigroup should be distinguished from
the syntactic monoid, defined to be the quotient of $X^*$ by the same
congruence.
A semigroup $S$ is said to be {\em locally idempotent and commutative} if,
for every idempotent $e \in S$, 
$eSe$ is an idempotent commutative subsemigroup of $S$. 

\begin{proposition}\label{prop:ltbool}
{\rm(1)}~\cite[Theorem 2.1(iv)]{BrzozowskiSimon}
$\LT$ is the Boolean closure of the languages of the form
$X^* w X^*$, $w X^*$, and $X^* w$, where $w \in X^*$.

\smallskip
\noindent
{\rm(2)}~\cite[Theorem 6.2]{BrzozowskiSimon},
\cite[Main Theorem, p.~63]{McNaughton}
$L$ is locally testable if and only if its
syntactic semigroup $\SSgp(L)$ is locally idempotent
and commutative.
\end{proposition}


Using Proposition \ref{prop:ltbool}, Brzozowski and Simon derived the following
characterization of
the locally testable languages. They stated it in terms of the associated
minimal automata and transition semigroups,
but we prefer to phrase it in terms of syntactic semigroups.

\begin{proposition}\cite[Theorem 6.2]{BrzozowskiSimon}  \label{prop:ltfsa}
A language $L$ is \klt{$k$} \ifof its syntactic semigroup $S$ satisfies the 
following:

For all $x,y,z \in X^*$ such that $l(x)=k-1$,
\begin{enumerate}
\item $xyxzx =_S xzxyx$, and
\item  whenever $xy=zx$, then 
$xy =_S xy^2$.
\end{enumerate}
\end{proposition}

We apply this characterization in the case $k=1$ in Section \ref{sec:1test}
below in order to characterize \klt{1} groups. More generally this
characterization provides a test for \klty{$k$}
that can be used algorithmically.


\section{\klt{1} groups}\label{sec:1test}

Given a word $w$ defined over $X$, we define the 
{\it support} of $w$, $\supp(w)$, to be the subset of 
$X$ consisting of the symbols in $X$ that appear in the word $w$.
A set $L$ of words over an alphabet $X$ is 
\klt{1} if membership of a word in 
$L$ can be determined 
simply from examination of which symbols appear in the word, 
irrespective of where they appear; that is, for some finite set 
${\mathcal S}$ of subsets of $X$, $w$ is in $L$ precisely 
if  $\supp(w) \in {\mathcal S}$. 

From Proposition \ref{prop:ltfsa} 
we see that a language $L$ is \klt{1} if and only if
the following two conditions hold in the syntactic
semigroup $S=\SSgp(L)$, for all representatives in $y,z \in X^*$: 
(1)~$yz =_S zy$  and (2)~$y=_S y^2$.
Note that there is no $x$ in these conditions because the $x$ of the conditions
of Proposition \ref{prop:ltfsa} has length 0 in this case.
This gives us the following.

\begin{corollary}\label{cor:1ltsyn}
A language $L$ is \klt{1} if and only if its
syntactic semigroup $\SSgp(L)$ is idempotent and commutative.
\end{corollary}

Using this result we can characterize \klt{1} groups.
\begin{theorem} \label{thm:1testfreeab}
Let $G$ be a finitely generated group.  Then the following are equivalent.
\begin{enumerate}
\item $G$ is free abelian.
\item $G$ is \klt{1}.
\item There is a finite symmetric generating set $X$ of $G$
such that the syntactic semigroup of $\Geo(G,X)$ is idempotent.
\end{enumerate}
\end{theorem}

\begin{proof}
Suppose that $G$ is finitely generated by $Y=\{y_1,\ldots y_n\}$ and
let $X=Y \cup Y^{-1}$.

If $G$ is free abelian on $Y$, then a word $w$ over $X$ is geodesic in $G$
\ifof its  support is a subset of one of the
$2^n$ distinct sets
\[\{y_1^{\epsilon_1},y_2^{\epsilon_2},\ldots,y_n^{\epsilon_n}\}\]
defined by sequences 
$(\epsilon_1,\epsilon_2,\ldots,\epsilon_n) \in \{1,-1\}^n$.
Hence $G$ is \klt{1}.

Corollary \ref{cor:1ltsyn} shows that if $G$ is \klt{1},
then (3) holds.

Finally suppose that the syntactic semigroup $S:=\SSgp(L)$ 
is idempotent, where $L := \Geo(G,X)$.
For any word $w$, we have $w =_S w^2$, and so (by
the definition of the congruence $\sim_L$) 
for all words $u,v \in X^*$ we have $uwv \in L \iff uw^2v \in L$.
In particular, $w^m \in L \iff w^{m+1} \in L$ for all positive 
integers $m$.
Hence if $w$ is geodesic, then $w^m$ must also be geodesic for all $m > 0$.
This implies in particular that $G$ must be torsion-free.

Now let
$x \in X$ be any generator, let $g \in G$ be any element, 
and let $w$ be a geodesic word
representing
$g^{-1}xg$. Then, since $w^m$ must be a geodesic representative
of $g^{-1}x^mg$ for all $m>1$, we see that $w$ must have length 1.
So $x$ has finitely many conjugates, and hence its centralizer has finite
index. Therefore the center $Z(G)$ has finite index in $G$.

Next we apply \cite[Thm 10.1.4]{Robinson} to deduce that 
the derived subgroup $G'$ must be finite. Since $G$ is also 
torsion-free this implies that $G'$ is trivial, and so $G$ is 
abelian. Thus $G$ is free abelian.
\end{proof}

We note that we have definitely made use of the group structure here.
In general, an idempotent syntactic semigroup is not
necessarily commutative.
As an example consider the \klt{2} language over $\{a,b\}$ consisting of
strings that do not start with $b$. Its syntactic semigroup $S$
is idempotent but $ab \neq_S ba$.
However, it is a consequence of the above theorem
that if the syntactic semigroup of the language of
geodesics of a group is idempotent, then the group is commutative,
and hence so is the syntactic semigroup.


\section{Closure under finite direct products}\label{sec:dirprod}

In this section we prove the following theorem.
\begin{theorem}\label{thm:ltprod}
For any $k > 0$, the class of \klt{$k$} groups
is closed under taking finite direct products.
\end{theorem} 

In fact the theorem is a special case of the following more general result.
For a class of languages $\mathcal{L}$, we understand an $\mathcal{L}$-group to be a 
group $G$ for which $\Geo(G,X)$ is in $\mathcal{L}$, for some finite symmetric
generating set $X$.

\begin{theorem}\label{thm:genprod}
Let $\mathcal{L}$ be a class of languages that is closed under the
operations of union and taking inverse images under length preserving morphisms.
Then the class of $\mathcal{L}$-groups is closed under taking finite direct products.
\end{theorem}

\begin{proof}

Suppose that the geodesic languages $L_1$ and $L_2$ of the groups $G$ and $H$
over the symmetric generating sets $X$ and $Y$ are in $\mathcal{L}$.
We adjoin generators to each of $X$ and $Y$ that are equal to the
identity elements of the respective groups; since any words
containing those generators are non-geodesic, this action does not change
$L_1$ or $L_2$, and hence it does not affect their membership in $\mathcal{L}$.

Define $Z:= \{(x,y): x \in X, y \in Y\} \subseteq G \times H$.
Then $Z$ generates $G \times H$ (since $X$ and $Y$ each contain an
identity). We define
$L$  to be the geodesic language for $G\times H$ over $Z$.

For any word $w$ over $Z$, define $\pi_1(w)$ to be the projection of $w$
onto $X^*$ and $\pi_2(w)$ to be its projection onto $Y^*$.
Then, for $w \in Z^*$, we see that if $w$ is to be non-geodesic, both
of its projections must be non-geodesic, and conversely. In other words,
the
following holds.
\[ w  \in L \iff \pi_1(w)  \in L_1 \vee \pi_2(w) \in L_2. \]

Since $\pi_1$ and $\pi_2$ are length preserving morphisms, 
the result now follows immediately.
\end{proof}

In order to deduce Theorem \ref{thm:ltprod} from Theorem \ref{thm:genprod}
we simply need the following lemma.
\begin{lemma}
The class of $k$-locally testable languages is closed under the
operations of union and taking inverse images under length preserving
morphisms.
\end{lemma}
\begin{proof}
Closure under union is immediate from the definition. 
The second claim follows from the fact that, if $\phi:A^* \rightarrow B^*$
is a length preserving morphism and $v,w \in A^*$ with $v  \sim_k w$,
then $\phi(v) \sim_k \phi(w)$
(where $\sim_k$ is as in the definition of local testability
in Section~\ref{sec:char}).
\end{proof}

Note that Theorem \ref{thm:genprod} can also be applied to the class of
star-free languages.
That result is also a  special case 
of \cite[Theorem 5.3]{HermillerHoltRees},
which is proved using a different construction, and different generating set.

At this point it is natural to ask whether the
class of groups with a locally testable language of
all geodesics is closed under the free product operation.
Although we do not have a definitive answer to that
question, we suspect that the answer is no, and
the following example shows that the natural
generating set does not work.

Consider the group 
$G={\Z}^2 * {\Z}=
(\langle a \rangle \times \langle b \rangle) *\langle c \rangle$.
Giving the ${\Z}^2$ subgroup the generating set from the proof of 
Theorem \ref{thm:genprod}, consider  
the generating set 
\begin{eqnarray*}
X & = & \{ (1,1),(1,b),(1,b^{-1}),(a,1),(a,b),(a,b^{-1}), \\
 & & (a^{-1},1),(a^{-1},b),(a^{-1},b^{-1}), c,c^{-1} \}
\end{eqnarray*}
for $G$.
For every natural number $n$, the language $\Geo(G,X)$ of geodesics
contains the word 
$u=(1,b)^n(a,b)^nc^n(a,b)^n(a,b^{-1})^{n}(a,b)^{n}$
but does not contain the word
$v=(1,b)^n(a,b)^{n}(a,b^{-1})^{n}(a,b)^nc^n(a,b)^{n}$.  However,
$u \sim_n v$.  So $\Geo(G,X)$ is not a union of $\sim_n$ equivalence
classes for any $n$, and hence is not locally testable.



\section{Locally testable geodesics and torsion 
conjugacy classes}\label{sec:lttorconj}

In this section we prove the following.

\begin{theorem} \label{thm:lttorconj}
A locally testable group
has finitely many conjugacy classes of torsion elements.
\end{theorem}

First we have a useful lemma.

\begin{lemma}\label{lem:cycl} 
Let $X$ be a set and let $k>0$ be a natural number.
There exists a natural number $N$ such that whenever $w$ 
is a word over $X$ of 
length at least $N$, then $w^2 \sim_k w^j$ for all $j \ge 2$.
Moreover, there exists a cyclic permutation $\widetilde w$
of $w$ satisfying $\widetilde w \sim_{k} \widetilde w^j$ for all $j\ge 1$. 
\end{lemma}

\begin{proof} 
Let $N:=2k(|X|^{2k}+1)$, and suppose that $w$ is a 
word with length $l(w) \ge N$.  
Since $l(w)>k$, for $j \ge 2$ we have 
$\pre_{k-1}(w) = \pre_{k-1}(w^2)=\pre_{k-1}(w^j)$,
$\suf_{k-1}(w) = \suf_{k-1}(w^2)=\suf_{k-1}(w^j)$
and $\sub_k(w^2)=\sub_k(w^j)$, as required.

By our choice of $N$, we can write $w=w_1 w_2 \cdots w_{|X|^{2k}+1}w'$,
where each subword $w_i$ has length $2k$.  There are only
$|X|^{2k}$ distinct words over $X$ of length $2k$, so for
some $i \neq j$, we have $w_i=w_j$.
Hence we can write $w=xuyuz$ for some word $u$ of length $2k$ and
$x,y,z \in X^*$.  Also write $u=st$ for words
$s,t \in X^*$ of length $k$ each.  Then $w=xstystz$.

Next let $\widetilde w := tystzxs$, a cyclic 
permutation of the word $w$.  By analyzing prefixes,
suffixes, and subwords of $\widetilde w$ and $\widetilde w^j$
as above, we find that
$\widetilde w \sim_{k} \widetilde w^j$ for all $j \ge 1$.
\end{proof}

Now we prove Theorem \ref{thm:lttorconj}.

\begin{proof}
Suppose that $X$ is a finite symmetric generating set for $G$
and that $L:=\Geo(G,X)$ is \klt{$k$}. 
Let $N:=2k(|X|^{2k}+1)$ as in the proof of Lemma \ref{lem:cycl}.
Suppose that $w$ is a word of length $l(w) \ge N$, and that 
$w$ represents a torsion element of $G$.

Lemma \ref{lem:cycl} says that there is a cyclic permutation
(and hence a conjugate) $\widetilde w$ of the word $w$ such that 
$\widetilde w \sim_{k} \widetilde w^i$ for all $i>0$. 
Then $\widetilde w \in L \iff \widetilde w^i \in L$, so
$\widetilde w$ is geodesic \ifof $\widetilde w^i$ is.
Since $w$ represents a torsion element, $\widetilde w$ 
must also be torsion in $G$, and so for some
$i$, and hence for all $i$, the word $\widetilde w^i$ is
not geodesic.  Therefore $\widetilde w =_G v$ 
for some word $v$ satisfying $l(v)<l(w)$ and
representing an element of $G$ in the conjugacy class
of $w$.  

If $l(v) \ge N$, then repeat this argument to obtain
words representing conjugates of $w$
of successively strictly shorter length.  Eventually
this process must end with a word $u$ of length 
$l(u) < N$ such that $u$ and $w$ are in the same
conjugacy class.  Therefore there are only finitely
many conjugacy classes of torsion elements.
\end{proof}


\section{Constructing a word acceptor for regular geodesics}\label{sec:geowa}

In this section, we describe a method for constructing a finite state
automaton that accepts the set of
all geodesic words in a finitely presented shortlex automatic group,
in the case that this set is a regular language
and hence such an automaton exists.  For Coxeter groups, a
method for doing this is described in~\cite{How93}, which 
is considerably faster than the general approach presented here. 

We suppose that our group $G$ is generated by the symmetric set $X$, and
as usual let $\Geo(G,X) \subset  X^*$ be the set of all geodesic words.
We also suppose that we have 
successfully computed the shortlex automatic structure
of $G$ with respect to some ordering of $X$. 
We shall assume throughout this section that $X^*$ is ordered by 
the associated shortlex ordering.

The procedure to be described here, which attempts to construct a finite
state automaton $\GW$ accepting a language $L(\GW)$ equal to $\geo$,  
will succeed eventually if $\geo$ is a regular language. 
If $G$ is word-hyperbolic, then the method described in Section 3
of~\cite{EH01} will construct $\GW$.  Otherwise, we (repeatedly, if
necessary) construct candidates for $\GW$, and try to prove their
correctness.
We shall first discuss our method for proving correctness of
$\GW$, which will succeed if and only and $\GW$ really is correct
(i.e. $L(\GW)=\geo$),
and then discuss how to come up with suitable candidates.

The automatic structure computed includes
the shortlex word acceptor $W$, the multiplier automata $M_x$ for
$x \in X \cup \{ \varepsilon \}$ (which we shall not need), and also
a {\em word-difference} automaton $D$ with the following properties:
\begin{mylist}
\item[(D1)] $(u,v)^+ \in L(D) \Rightarrow u =_G v$;
\item[(D2)] If $u,v \in X^*$ with $u,v \in L(W)$,
$x \in X \cup \{ \varepsilon \}$ and $ux =_G v$, then $(ux,v)^+ \in L(D)$.
\end{mylist}
Here $(u,v)^+$ denotes the padded pair corresponding to $u,v \in X^*$.
In fact, the construction of $D$ is such that its start state $\sigma$
is its only accepting state, and $D$ contains
transitions labeled $(a,a)$ from $\sigma$ to $\sigma$ for all
$a \in X$.  It therefore has the additional property:
\begin{mylist}
\item[(D3)] If $(u,v)^+ \in L(D)$ and $w,w' \in X^*$
then $(wu,wv)^+ \in L(D)$ and, if $l(u) = l(v)$, then $(wuw',wvw') \in L(D)$.
\end{mylist}

For a candidate $\GW$ for a finite state automaton with language
$\geo$, we use standard operations on finite state automata, as
described in Chapter 1 of~\cite{ECHLPT} or Section 13.1 of~\cite{H05},
to check whether $\GW$ satisfies the hypotheses of the following theorem.
If so, then the theorem tells us that $\GW$ is correct (that is,
$L(\GW) = \geo$). It is easy to see that
these conditions are also necessary for the correctness of $\GW$, so
our procedure will succeed if and only if $\GW$ is correct.

\begin{theorem} Let $\GW$ be a finite state automaton over $X$ which
satisfies the following conditions:
\begin{mylist}
\item[(i)] $L(\GW)$ is prefix-closed;
\item[(ii)] $L(W) \subseteq L(\GW)$;
\item[(iii)] If $(u,v)^+ \in L(D)$ with $l(u) = l(v)$, then $u \in L(\GW)
\iff v \in  L(\GW)$;
\item[(iv)] If $(u,v)^+ \in L(D)$ with $l(u) > l(v)$, then $u \not\in L(\GW)$.
\end{mylist}
Then $L(\GW) = \geo$.
\end{theorem}
\begin{proof} Suppose that $\geo \not\subseteq L(\GW)$ and let
$w$ be the shortlex least element of $\geo \setminus L(\GW)$.
Then, by (ii), $w \not\in L(W)$. Let $w = w'uw''$, where $u$ is of minimal
length with $u \not\in L(W)$. Then, by (D2), we have $(u,v)^+ \in L(D)$ for
some $v \in L(W)$. Since by (D1) $u =_G v$, and $w \in \geo$,
we must have $l(u)=l(v)$ and hence, by (D3), $(w'uw'',w'vw'') \in L(D)$.
Moreover $w'vw'' \in \geo$.
Now, since $v \in L(W)$, we have $u > v$ (recall that we are using the shortlex
ordering on $X^*$) and hence $w = w'uw'' > w'vw''$,
and then $w'vw'' \in L(\GW)$ by choice of $w$. But this contradicts (iii).

Suppose, on the other hand, that $L(\GW) \not\subseteq \geo$ and let
$w$ be the shortlex least element of $L(\GW) \setminus \geo$.
Then $w \not\in \geo$ implies $w \not\in L(W)$, and again we can write
$w = w'uw''$, where $u$ is of minimal length with $u \not\in L(W)$.
Again we have $(u,v)^+ \in L(D)$ for some $v \in L(W)$.
If $l(u) > l(v)$ then, by (i) and the minimality of $w$, we must have $w = w'u$,
and hence, by (D3), $(w'u,w'v)^+ \in L(D)$. But this contradicts (iv).
On the other hand, if $l(u) = l(v)$ then, by (D3), $(w'uw'',w'vw'') \in L(D)$,
which contradicts (iii) again.
\end{proof}

This gives the method of testing, for a candidate automaton $\GW$, whether
$L(\GW) = \geo$.
Now we turn to the problem of constructing potential candidates for $\GW$.
 
In the methods to be described below, we can prove that if $\geo$ is
regular and if we run 
the relevant programs for sufficiently long, then $\GW$ will be correctly
calculated. Although we have no means of estimating for how long we
actually need to run the programs, we can test a sequence of candidates
for correctness, and thereby produce a terminating algorithm
to compute $\GW$ with $L(\GW) = \geo$, provided of course that
$\geo$ is regular.

As mentioned previously,
if $G$ is word-hyperbolic, then the procedure described in Section 3
of~\cite{EH01} will construct the correct automaton $\GW$.
This method can be summarized as follows. Starting from the automaton
$\GW_0 := W$, the shortlex word acceptor in the automatic structure for
$G$, we construct the automaton $\GW_1$ which accepts all words $v$ for
which there exists a $w \in L(\GW_0)$ with $l(w) = l(v)$ and with $(v,w)$ in
the language of the word-difference machine $D$. So $L(\GW_0) \subseteq
L(\GW_1) \subseteq \geo$. We can construct $\GW_i$ from $\GW_{i-1}$ in
the same way, for all $i > 0$. It can be shown that, if $L(\GW_i) =
L(\GW_{i-1})$ for some $i$, then $L(\GW_i) = \geo$, and that this
is the case if and only $G$ is word-hyperbolic.

If applied to a group that is not word-hyperbolic, then this procedure
will not terminate, and will construct a sequence $\GW_i$ ($i \in \N$) of
automata with $L(\GW_i) \subset \geo$, $L(\GW_i) \subset
L(\GW_{i+1})$, and $\geo = \cup_{i=1}^{\infty} L(\GW_i)$. So, for
any $k > 0$, there exists an $n$ such that, for all $i \ge n$, 
$L(\GW_i)$ has the following property:
\begin{mylist}
\item[($P_k$)] the set of
words of length at most $k$ in $L(\GW_i)$ is equal to the set of words of
length at most $k$ in $\geo$.
\end{mylist}

We apply a method described by
Trakhtenbrot and Barzdin in~\cite[Section IV.2]{TB73}
to construct candidates $\GW$ for an automaton that
satisfies $L(\GW) = \geo$.
This procedure, as presented in~\cite{TB73},
assumes the existence of an automaton $M$, and
takes as input a set of pairs $(w,\mathsf{b})$, where $w$ is a word
and $\mathsf{b}$ is $\mathsf{true}$ or $\mathsf{false}$,
with $\mathsf{b}=\mathsf{true}$  if and only if $w \in L(M)$.
It attempts to construct $M$ from this information.
It is proved in~\cite[Theorem 2.16]{TB73} that,
if $M$ is known to have at most $n$ states and the procedure is given the
membership of all words of length at most $2n-1$ in $L(M)$ as input,
then it will successfully construct $M$.

We shall apply a variant of this procedure using an automaton $M'$ and an
integer $k > 0$ as input, and use the membership of words of length at most
$k$ in $L(M')$ as our criterion to determine whether they lie in $L(M)$.
It will either abort or output a new automaton $c(M',k)$;
specifically, we apply it to $GW_i$, for some $i$ and $k$.

It follows from~\cite[Theorem 2.16]{TB73} that,
for any finite state automaton $M$,
there exists an integer $B(M) >0$, such that if $k \ge B(M)$ and
the sets of words of length at most $k$ in $L(M)$ and $L(M')$ are
equal, then the procedure applied to $M'$ will output $M$ as $c(M',k)$.
If $M$ has $n$ states, then this is true with $B(M)=2n-1$,
but it is proved in~\cite[Theorem 5.10]{TB73}
that there is a constant $C$ such that,
for almost all automata $M$, $B(M) \le C \log _{|A|} (n)$, where $A$
is the alphabet of $M$.

Unfortunately, we know no upper bound on the number of states
of $M$, so we have no way of estimating $B(M)$ in advance.
However, since we know that, for any $k$,
$P_k$ is satisfied by $\GW_i$ for all sufficiently
large $i$, the following is true.
If $\geo$ is a regular language and if
we calculate $\GW_i$ for sufficiently large $i$
and then apply the procedure for sufficiently large $k$, then
an automaton $\GW$ with $L(\GW) = \geo$
will be output as $c(\GW_i,k)$.

Here is a summary of our version of the Trakhtenbrot and Barzdin procedure.
For a state $\sigma$ of a deterministic finite state automaton $M$ with start
state $\sigma_0$, we define the {\em depth} of $\sigma$ to be the length of a
shortest word $w$ with $\sigma_0^w = \sigma$ (where $\sigma_0^w$ is the
state reached from $\sigma_0$ on reading the string $w$). For an
integer $k>0$, define
$M(\sigma,k)$ to be the set of words $w$ with $l(w) \le k$ such that $\sigma^w$
is an accept state of $M$.
For two states $\sigma,\tau$ of $M$, let $d$ be maximum of the depths of
$\sigma$ and $\tau$ and, for $k > d$, define $\sigma$ and $\tau$ to be
{\em $k$-equivalent} if $M(\sigma,k-d) = M(\tau,k-d)$.
Note that, since the value of $d$ depends on the pair of states
$\sigma,\tau$, $k$-equivalence is not necessarily an equivalence relation
on the set of all states of $M$, although it will be for sufficiently
large $k$.

With input $M'$ and $k$, the procedure attempts to define
an automaton $c(M',k)$  in which the set of states is the set of
$k$-equivalence classes of states of $M'$, and the transitions are induced
from those of $M'$. If $k$-equivalence turns out not to be a
equivalence relation, or if transitions with a given label from
$k$-equivalent states do not lead to $k$-equivalent states,
then the procedure aborts.  Otherwise it outputs $c(M',k)$.

Unfortunately, we do not know $B(\GW)$ in advance, and neither do we know
how large $i$ must be to guarantee that $\GW_i$ satisfies $P_k$. We have
not yet implemented any heuristics for estimating these values; we have simply
guessed at them.

We finish with some examples in which we have successfully computed $\GW$
with $L(\GW) = \geo$, together with the values of $k$ and $i$ used,
and the number of states of $\GW$. The first example is a 5-generator
presentation of the wreath product
of an infinite cyclic group with the group of order 2. There are
presentations on fewer generators with regular sets of geodesics, but
with this presentation $\geo$ also 
turns out to be 2-locally testable.
The second, third and fourth examples are 2-generator Artin groups and
they each have locally testable sets of geodesics. 
Indeed, we checked, using the conditions in Proposition~\ref{prop:ltfsa},
that they are $k$-locally testable for $k = 3,4$ and 5, respectively,
a fact which we know to be true from~\cite[Proposition 4.3]{MM06}.

The final example is a 4-generator Coxeter group for which
$\geo$ is regular but not locally testable; 
in fact it is not even star-free.
As mentioned at the beginning of this section, it is more 
efficient to use the method
described in~\cite{How93} to compute $\GW$ for this example. 

\smallskip\noindent
{\bf 1.}
$G = \langle\,a,b,t,u,v \mid ab \te ba,\, t^2 \te 1,\, tat \te b,\, at \te u,\,
 bt\te v\,\rangle$;
$i=3$, $k=4$, 10 states. 

\smallskip\noindent
{\bf 2.} $G = \langle\, a,b \mid aba=bab\,\rangle$;
$i=3$, $k=7$, 28 states. 

\smallskip\noindent
{\bf 3.} $G = \langle\, a,b \mid abab=baba\,\rangle$;
$i=4$, $k=10$, 61 states. 

\smallskip\noindent
{\bf 4.} $G = \langle\, a,b \mid ababa=babab\,\rangle$;
$i=5$, $k=12$, 115 states. 

\smallskip\noindent
{\bf 5.} $G  \te  \langle\, a,b,c,d \mid  a^2 \te b^2 \te c^2 \te d^2 \te 
(ab)^3 \te (bc)^3 \te (cd)^3 \te (da)^3 \te (ac)^2 \te (bd)^2 \te 1\,\rangle$;
$i=6$, $k=14$, 125 states.


\end{document}